\documentclass[a4paper,12pt]{article}
\usepackage[top=2.5cm,bottom=2.5cm,left=2.5cm,right=2.5cm]{geometry}
\usepackage{cite, amsmath, amssymb}

\usepackage{graphicx}
\usepackage{color, colortbl}

\newtheorem{thm}{Theorem}
\newtheorem{cor}{Corollary}
\newtheorem{lem}{Lemma}
\newtheorem{rmk}{Remark}

\newenvironment{proof}{\par \noindent \textbf{Proof: }}{\QED \par \bigskip \par}
\newcommand{\QED}{\hfill$\square$}

\newcommand{\al}{\alpha}

\allowdisplaybreaks[4]

\begin{document}

\baselineskip=0.30in


\begin{center}

{\LARGE \bf Estimating Some General Molecular Descriptors of Saturated Hydrocarbons}
\vspace{10mm}

{\large \bf Akbar Ali$^{\dag}$, Zhibin Du$^{\ddag}$, Kiran Shehzadi$^{\dag}$}

\vspace{6mm}

\baselineskip=0.20in

$^{\dag}${\em Knowledge Unit of Science, University of Management \& Technology}\\
{\em Sialkot, Pakistan}\\
E-mail: {\tt akbarali.maths@gmail.com, kiranshabbir43@gmail.com}\\

\vspace{2mm}

$^{\ddag}${\em School of Mathematics and Statistics, Zhaoqing University}\\
{\em Zhaoqing 526061, Guangdong, P.R. China}\\
E-mail: {\tt zhibindu@126.com}

\makeatother


\vspace{6mm}


\end{center}

\vspace{6mm}

\baselineskip=0.23in

\noindent {\bf Abstract.}
Three general molecular descriptors, namely the general sum-connectivity index, general Platt index and ordinary generalized geometric-arithmetic index, are studied here. Best possible bounds for the aforementioned descriptors of arbitrary saturated hydrocarbons are derived. These bounds are expressed in terms of number of carbon atoms and number of carbon-carbon bonds of the considered hydrocarbons.\\[2mm]



\baselineskip=0.30in

\vspace{4mm}

\section{Introduction}

Molecules can be represented by graphs (usually known as molecular graphs) in which vertices correspond to the atoms while edges represent the covalent bonds between atoms \cite{Trinajstic92,Gutman86}. All the graphs, considered in the present study, are hydrogen-depleted molecular graphs representing saturated hydrocarbons.
According to Todeschini and Consonni \cite{ToCo2000} ``\textit{molecular descriptor} is the final result of a logical and mathematical procedure which transforms chemical information encoded within a symbolic representation of a molecule into an useful number or the result of some standardized experiment''. A molecular descriptor calculated from a molecular graph is simply known as a \textit{topological index} \cite{Trinajstic92,Gutman86}. In the quantitative structure-property relationship studies, topological indices are often used to model the physicochemical properties of molecules \cite{Atabati17,Mozrzymas17,Santos17,Shafiei17}.

The Platt index ($Pl$), which was proposed for predicting paraffin properties \cite{Platt52}, is one of the oldest topological indices. This index is defined as:
\[Pl(G)=\sum_{uv\in E(G)}(d_{u}+d_{v}-2),\]
where $uv$ is the edge connecting the vertices $u,v$ of the (molecular) graph $G$, $E(G)$ is the edge set of $G$, and $d_{u}$ is the degree of the vertex $u$. It should be mentioned here that the Platt index can be written as
$$
Pl(G)=M_{1}(G)-2m,
$$
where $m$ is the number of edges in the graph $G$ and $M_{1}$ is the first Zagreb index, appeared in 1972 within the study of total $\pi$-electron energy of alternant hydrocarbons \cite{Gutman72}. The first Zagreb index can be defined \cite{Doslic11} as:
\[M_{1}(G)=\sum_{uv\in E(G)}(d_{u}+d_{v}).\]
Mathematical properties related to the first Zagreb index (and hence related to the Platt index) can be found in the recent surveys \cite{Borovicanin-17,Borovicanin-MCM19,Ali18} and related references listed therein.

The connectivity index (also known as branching index and Randi\'{c} index) \cite{r3} is one of the most studied and applied topological indices, which was proposed in 1975 for measuring the extent of branching of the carbon-atom skeleton of saturated hydrocarbons. The connectivity index for a graph $G$ is defined as
\[R(G)=\displaystyle\sum_{uv\in E(G)}(d_{u}d_{v})^{-\frac{1}{2}}.\]
Details about the chemical applicability and mathematical properties of this index can be found in the survey \cite{Li08}, recent papers \cite{Ali17R,Cui17,Li16,Mansour17,Das-17,Gutman-18} and related references listed therein.

Several modified versions of the connectivity index were appeared in literature. One of such modified versions is the sum-connectivity index, defined as:
\[\chi(G)=\sum_{uv\in E(G)}(d_{u}+d_{v})^{-\frac{1}{2}},\]
which was proposed in 2009, by Zhou and Trinajsti\'{c} \cite{Zhou09}. After that, Zhou and Trinajsti\'{c} \cite{Zhou10} introduced the following generalization of the sum-connectivity index and first Zagreb index:
\[\chi_{\alpha}(G)=\sum_{uv\in E(G)}(d_{u}+d_{v})^{\alpha},\]
where $\alpha$ is a non-zero real number. Chemical applicability of the sum-connectivity index was in \cite{Lucic-13,Lucic-15,Lucic-Trinajstic-Zhou-CPL-09,Vukcevic-Trinajstic-CCA-11}. We recall that $\chi_{2}$ is the hyper-Zagreb index \cite{Shirdel13} and $H(G) = 2\chi_{-1}(G)$, where $H$ is the harmonic index \cite{Fajtlowicz87}. The general sum-connectivity index has attracted a considerable attention from researchers, see (for example) the recent survey \cite{Ali-MATCH-19}, recent papers \cite{Ali-Dimitov-2018,Zhu16,Cui-17,Tomescu16,Ali17,Ali-18,Wang17,Ramane17,Akhter17,Arshad17,Jamil17} and related references listed therein.

Recently, the general Platt index ($Pl_{\alpha}$) was proposed in \cite{Ali-18}, which is defined as:
\[Pl_{\alpha}(G)=\sum_{uv\in E(G)}(d_{u}+d_{v}-2)^{\alpha},\]
where $\alpha$ is a non-zero real number. We recall that $Pl_2$ coincides with the reformulated first Zagreb index $EM_1$, introduced in \cite{Mili04}.
The ordinary generalized geometric-arithmetic index for a molecular graph $G$ is defined \cite{Eliasi11} as
\begin{equation}\label{Eq0-GA}
\mbox{OGA}_{k}(G)=\sum_{uv\in E(G)} \left(\frac{2\sqrt{d_u d_v}}{d_u+d_v} \right)^{k},
\end{equation}
where $k$ is any positive real number.

A graph with $n$ vertices and $m$ edges is known as an $(n,m)$-graph. A graph in which every vertex has degree at most 4 is called a molecular graph. Undefined notations and terminologies from (chemical) graph theory can be found in \cite{Harary69,Trinajstic92,Gutman86,Bondy08}.

The main purpose of the present paper is to derive the sharp lower bounds (for $-1\le \al<0$ and for $0<k\le 1$) and sharp upper bounds (for $0< \al\le2$) on the topological indices $\chi_{\alpha}$, $Pl_\alpha$ and $\mbox{OGA}_{k}$ for molecular $(n,m)$-graphs.

\section{Main Results}\label{Sec3}



Let $n_{i}(G)$ (or simply $n_i$) be the number of vertices of degree $i$ in a graph $G$. Denote by $x_{i,j}(G)$ (or simply by $x_{i,j}$) the number of edges in a graph $G$ connecting the vertices of degrees $i$ and $j$. The general sum-connectivity index for any molecular $(n,m)$-graph $G$ can be rewritten as:
\begin{equation}\label{Eq0}
\chi_{\alpha}(G)=\sum_{1\le i \le j \le 4} x_{i,j}(i+j)^{\alpha}.
\end{equation}
Also, the following system of equations holds for any molecular $(n,m)$-graph $G$:
\begin{equation}\label{Eq555b}
\sum_{i=1}^{4}n_{i}=n,
\end{equation}
\begin{equation}\label{Eq556b}
\sum_{i=1}^{4}i\cdot n_{i}=2m,
\end{equation}
\begin{equation}\label{Eq557b}
\sum_{ \substack{ 1\leq i\leq 4 \\
         i\neq j}}x_{j,i}+2x_{j,j}=j\cdot n_{j}
\end{equation}
where $j=1,2,3,4$.

The following values of $x_{1,4}$ and $x_{4,4}$ can be obtained by solving the system of Eqs. (\ref{Eq555b})-(\ref{Eq557b}) (see also  \cite{Gutman2000}):
\[x_{1,4}=\frac{4n}{3}-\frac{2m}{3}-\frac{4}{3}x_{1,2}-\frac{10}{9}x_{1,3}-\frac{2}{3}x_{2,2}
-\frac{4}{9}x_{2,3}-\frac{1}{3}x_{2,4}-\frac{2}{9}x_{3,3}-\frac{1}{9}x_{3,4},\]
\[x_{4,4}=-\frac{4n}{3}+\frac{5m}{3}+\frac{1}{3}x_{1,2}+\frac{1}{9}x_{1,3}-\frac{1}{3}x_{2,2}
-\frac{5}{9}x_{2,3}-\frac{2}{3}x_{2,4}-\frac{7}{9}x_{3,3}-\frac{8}{9}x_{3,4}.\]
After substituting the values of $x_{1,4}$ and $x_{4,4}$ in Eq. (\ref{Eq0}), one has:
\begin{eqnarray}\label{Eq557bc}
\chi_\al(G)&=& \frac{4}{3}(5^\al-8^\al)n-\frac{1}{3}(2\cdot 5^\al-5\cdot 8^\al)m + x_{1,2}\Theta_{1,2}
+x_{1,3}\Theta_{1,3} + x_{2,2}\Theta_{2,2}\nonumber \\
&& + x_{2,3}\Theta_{2,3} +x_{2,4}\Theta_{2,4}+x_{3,3}\Theta_{3,3}+x_{3,4}\Theta_{3,4},
\end{eqnarray}
where
$$
\Theta_{1,2}= 3^\al -\frac{4}{3}\cdot 5^\al +\frac{1}{3}\cdot 8^\al, \ \
\Theta_{1,3} = 4^\al -\frac{10}{9}\cdot 5^\al +\frac{1}{9}\cdot 8^\al,
$$
$$
\Theta_{2,2}= 4^\al -\frac{2}{3}\cdot 5^\al -\frac{1}{3}\cdot 8^\al, \ \
\Theta_{2,3}= \frac{5}{9} \left( 5^\al - 8^\al \right),
$$
$$
\Theta_{2,4}= 6^\al -\frac{1}{3}\cdot 5^\al -\frac{2}{3}\cdot 8^\al, \ \
\Theta_{3,3}= 6^\al -\frac{2}{9}\cdot 5^\al -\frac{7}{9}\cdot 8^\al,
$$
$$
\Theta_{3,4}= 7^\al -\frac{1}{9}\cdot 5^\al -\frac{8}{9}\cdot 8^\al.
$$
Setting
\begin{equation}\label{Eq-K-00}
\Gamma(G)= x_{1,2}\Theta_{1,2}+x_{1,3}\Theta_{1,3} + x_{2,2}\Theta_{2,2} + x_{2,3}\Theta_{2,3} +x_{2,4}\Theta_{2,4}+x_{3,3}\Theta_{3,3}+x_{3,4}\Theta_{3,4},
\end{equation}
and Eq. (\ref{Eq557bc}) yields
\begin{equation}\label{Eq7}
\chi_\al(G)= \frac{4}{3}(5^\al-8^\al)n-\frac{1}{3}(2\cdot 5^\al-5\cdot 8^\al)m + \Gamma(G).
\end{equation}

From Figures \ref{f1} and \ref{f2}, it is clear that $\Theta_{1,2}, \Theta_{1,3}, \Theta_{2,2}, \Theta_{2,3}, \Theta_{2,4}, \Theta_{3,3}$ and $\Theta_{3,4}$ are all positive for $-1\le \al <0$ and all negative for $0< \al \le 2$, and hence from Eq. (\ref{Eq-K-00}) it follows that
\[
\Gamma(G) \begin{cases}
\ge 0 & \text{for $-1\le\alpha<0$,}\\
\le 0 & \text{for $0<\alpha\le2$,}
\end{cases}
\]
with either equality if and only if $x_{1,2}=x_{1,3}=x_{2,2} =x_{2,3} =x_{2,4}=x_{3,3}=x_{3,4}=0$.
Therefore, from Eq. (\ref{Eq7}), we have the following result.

\begin{figure}[h]
  \centering
   \includegraphics[width=6.3in, height=4.6in]{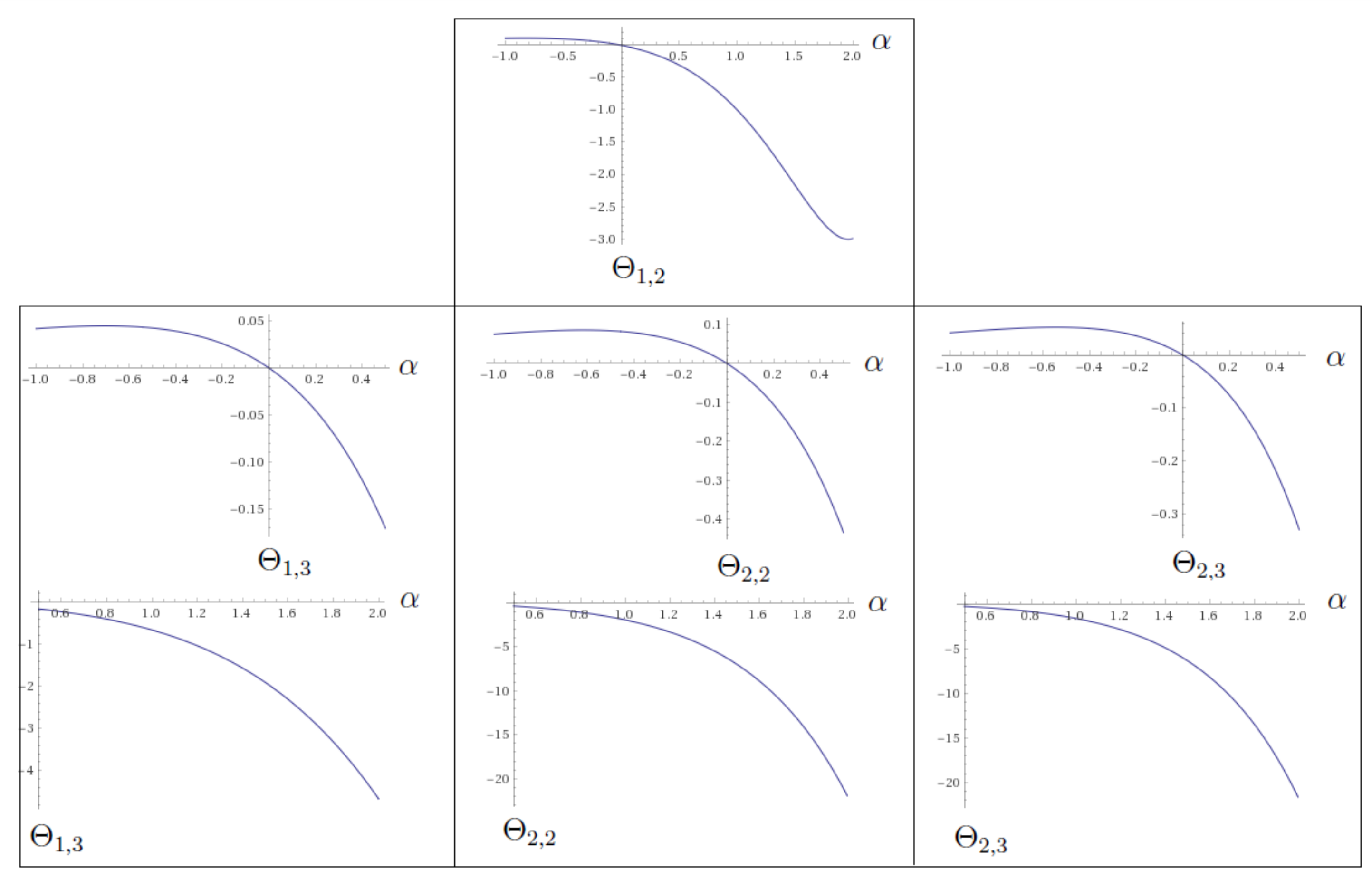}
    \caption{The graphs of $\Theta_{1,2}, \Theta_{1,3}, \Theta_{2,2}$ and $\Theta_{2,3}$ for $-1\le \al \le 2$.}
    \label{f1}
\end{figure}

\begin{figure}[h]
   \centering
   \includegraphics[width=6.3in, height=3.1in]{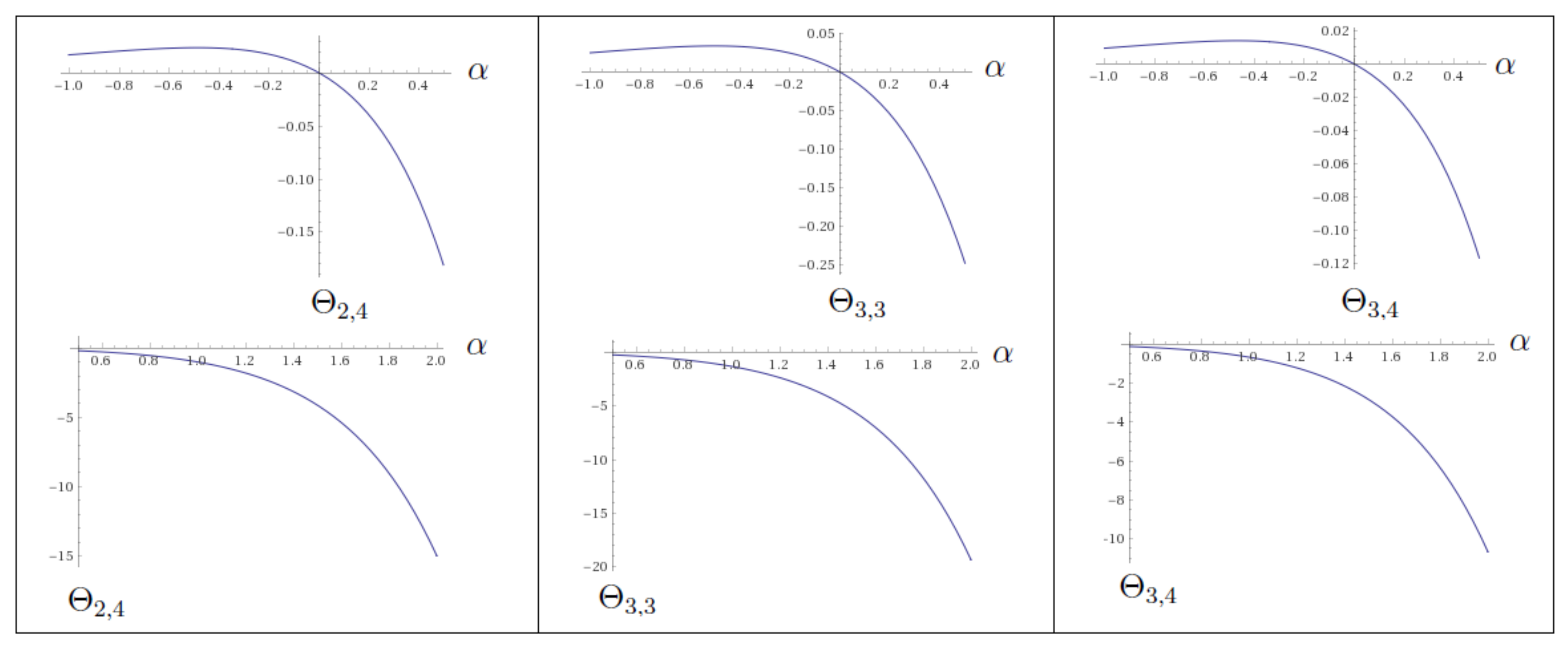}
    \caption{The graphs of $\Theta_{2,4}, \Theta_{3,3}$ and $\Theta_{3,4}$ for $-1\le \al \le 2$.}
     \label{f2}
 \end{figure}

\begin{thm}\label{thm-bd}
If $G$ is a molecular $(n,m)$-graph, $n\ge5$, then
\[
\chi_\al(G) \begin{cases}
\ge \frac{4}{3}(5^\al-8^\al)n-\frac{1}{3}(2\cdot 5^\al-5\cdot 8^\al)m & \text{for $-1\le\alpha<0$,}\\
\le \frac{4}{3}(5^\al-8^\al)n-\frac{1}{3}(2\cdot 5^\al-5\cdot 8^\al)m & \text{for $0<\alpha\le2$,}
\end{cases}
\]
with either equality if and only if $G$ contains no vertices of degrees $2$ and $3$.
\end{thm}

The next result follows from the facts $M_1 = \chi_1$, $\chi = \chi_{-1/2}$, $H = 2\chi_{-1}$ and Theorem \ref{thm-bd}.
\begin{cor}\label{cor-bd}
If $G$ is a molecular $(n,m)$-graph, $n\ge5$, then the following inequalities hold:
\[
M_1(G)\le 10m-4n,
\]
\[
H(G)\ge \frac{4n+3m}{20},
\]
\[
\chi(G)\ge \frac{4}{3}\left(\frac{1}{\sqrt{5}}-\frac{1}{2\sqrt{2}}\right)n+\frac{1}{3}\left(\frac{5}{2\sqrt{2}}-\frac{2}{\sqrt{5}}\right)m.
\]
The equality in any of the above inequalities holds if and only if $G$ does not contain vertices of degrees 2 and 3.
\end{cor}

\begin{rmk}
In case of trees, one has $m=n-1$ and hence the first (second, respectively) inequality of Corollary \ref{cor-bd} generalizes the bound reported in \cite{Vukicevic14} (\cite{Liu14}, respectively) to any $(n,m)$-graph. The last inequality of Corollary \ref{cor-bd} was derived also in \cite{Zhou09}. Moreover, a sharp upper bound on the hyper Zagreb index $\chi_2$ also follows from Theorem \ref{thm-bd}.
\end{rmk}

We note that the bounds given in Theorem \ref{thm-bd} are sharp when $m+n\equiv0$ (mod  3). Now, we improve the aforementioned bounds (given in Theorem \ref{thm-bd}) when $m+n\equiv1$ (mod  3) and when $m+n\equiv2$ (mod  3).


In what follows, we will take $\al\ne 1$, because the problem of finding sharp upper bound of $\chi_1$ for molecular $(n,m)$-graphs has already been solved in \cite{Hu05} (see also \cite{Vukicevic14} for $m=n-1$).

The following lemma can be verified by Mathematica easily.

\begin{lem} \label{Lem-K-01a}
The functions $\Theta_{1,2}, \Theta_{1,3}, \Theta_{2,2}, \Theta_{2,3}, \Theta_{2,4}, \Theta_{3,3}$ and $\Theta_{3,4}$, given in Eq. (\ref{Eq557bc}), satisfy the following inequalities:
\begin{enumerate}
\item[$(i)$]
\[
\begin{cases}
\min \{ \Theta_{1,3},\Theta_{2,3}, \Theta_{3,3} \}  > \Theta_{3,4} > 0  & \text{for $-1\le\alpha<0$,}\\
\max \{ \Theta_{1,3}, \Theta_{2,3}, \Theta_{3,3} \} < \Theta_{3,4} <0   & \text{for $0<\alpha<1$,}\\
\max \{ \Theta_{2,3}, \Theta_{3,3} , \Theta_{3,4} \} < \Theta_{1,3} <0 & \text{for $1<\alpha\le2$;}
\end{cases}
\]
\item[$(ii)$]
\[
\begin{cases}
\min\{ \Theta_{1,2}, \Theta_{2,2}, \Theta_{2,3} \} > \Theta_{2,4}>0 & \text{for $-1\le\alpha<0$,}\\
\max \{ \Theta_{1,2},\Theta_{2,2} , \Theta_{2,3} \} < \Theta_{2,4}<0  & \text{for $0<\alpha<1$,}\\
\max \{ \Theta_{2,2} , \Theta_{2,3} , \Theta_{2,4} \} < \Theta_{1,2}<0 & \text{for $1<\alpha\le2$.}
\end{cases}
\]
\end{enumerate}
\end{lem}

\begin{lem} \label{Lem-K-01}
If a molecular graph $G$ satisfies the inequality $n_2+n_3 \ge 2$, then for the invariant $\Gamma (G)$, defined in Eq. (\ref{Eq-K-00}), it holds that
\[
\Gamma(G) \begin{cases}
> 2\Theta_{2,4}  & \text{for $-1\le\alpha<0$,}\\
< 2\Theta_{2,4}  & \text{for $0<\alpha<1$,}\\
< 2\Theta_{1,3}+\Theta_{3,4}   & \text{for $1<\alpha\le2$.}
\end{cases}
\]

\end{lem}

\begin{proof}
First we consider when at least one of $x_{2,2}, x_{2,3}, x_{3,3}$ is positive.

If $x_{2,2}\ge1$, then from Eq. (\ref{Eq-K-00}) and verification by using Mathematica, we have
\[
\Gamma(G) \begin{cases}
\ge \Theta_{2,2} >  2\Theta_{2,4} & \text{for $-1\le\alpha<0$,}\\
\le \Theta_{2,2} < 2\Theta_{2,4} & \text{for $0<\alpha<1$,}\\
\le \Theta_{2,2} < 2\Theta_{1,3}+\Theta_{3,4} & \text{for $1<\alpha\le2$.}
\end{cases}
\]

Suppose that $x_{2,3}\ge1$. From Eq. (\ref{Eq-K-00}) and Lemma \ref{Lem-K-01a} (i), it follows that
\begin{eqnarray} \label{eq-add}
\Gamma(G) \begin{cases}
\ge x_{2,3} \Theta_{2,3} + (x_{1,3}+x_{3,3}+x_{3,4})\Theta_{3,4}  & \text{for $-1\le\alpha<0$,}\\
\le x_{2,3} \Theta_{2,3} + (x_{1,3}+x_{3,3}+x_{3,4})\Theta_{3,4}  & \text{for $0<\alpha<1$,}\\
\le x_{2,3} \Theta_{2,3} + (x_{1,3}+x_{3,3}+x_{3,4})\Theta_{1,3}  & \text{for $1<\alpha\le2$.}
\end{cases}
\end{eqnarray}
From Eq. (\ref{Eq557b}) with $j = 3$, we have
$$
x_{1,3} + x_{2,3} + 2x_{3,3} + x_{3,4} = 3 n_3.
$$
Note that $n_3 \ge 1$. So there are two cases need to be considered:
\begin{itemize}
\item
$x_{2,3} = 1$ or $2$, and $x_{1,3} + x_{3,3} + x_{3,4} \ge 1$;
\item
$x_{2,3} \ge 3$.
\end{itemize}
If $x_{2,3} = 1$ or $2$, and $x_{1,3} + x_{3,3} + x_{3,4} \ge 1$, then by (\ref{eq-add}) and Mathematica,
\[
\Gamma(G) \begin{cases}
\ge   \Theta_{2,3}+ \Theta_{3,4} >  2\Theta_{2,4} & \text{for $-1\le\alpha<0$,}\\
\le  \Theta_{2,3}+ \Theta_{3,4} < 2\Theta_{2,4}  & \text{for $0<\alpha<1$,}\\
\le  \Theta_{2,3}+ \Theta_{1,3} < 2\Theta_{1,3}+\Theta_{3,4} & \text{for $1<\alpha\le2$.}
\end{cases}
\]
If $x_{2,3} \ge 3$, then by (\ref{eq-add}) and Mathematica,
\[
\Gamma(G) \begin{cases}
\ge 3\Theta_{2,3} >  2\Theta_{2,4} & \text{for $-1\le\alpha<0$,}\\
\le 3\Theta_{2,3} <  2\Theta_{2,4}  & \text{for $0<\alpha<1$,}\\
\le 3\Theta_{2,3} < 2\Theta_{1,3}+\Theta_{3,4} & \text{for $1<\alpha\le2$.}
\end{cases}
\]

%
%
%

Suppose that $x_{3,3}\ge1$.
From Eq. (\ref{Eq-K-00}) and Lemma \ref{Lem-K-01a} (i), it follows that
\begin{eqnarray} \label{eq-add1}
\Gamma(G) \begin{cases}
\ge x_{3,3} \Theta_{3,3} + (x_{1,3}+x_{2,3}+x_{3,4})\Theta_{3,4}  & \text{for $-1\le\alpha<0$,}\\
\le x_{3,3} \Theta_{3,3} + (x_{1,3}+x_{2,3}+x_{3,4})\Theta_{3,4}  & \text{for $0<\alpha<1$,}\\
\le x_{3,3} \Theta_{3,3} + (x_{1,3}+x_{2,3}+x_{3,4})\Theta_{1,3}  & \text{for $1<\alpha\le2$.}
\end{cases}
\end{eqnarray}
From Eq. (\ref{Eq557b}) with $j = 3$, we have
$$
x_{1,3} + x_{2,3} + 2x_{3,3} + x_{3,4} = 3 n_3.
$$
Note that $n_3 \ge 2$. So there are two cases need to be considered:
\begin{itemize}
\item
$x_{3,3} = 1$ or $2$, and $x_{1,3} + x_{2,3} + x_{3,4} > 1$;
\item
$x_{3,3} \ge 3$.
\end{itemize}
If $x_{3,3} = 1$ or $2$, and $x_{1,3} + x_{2,3} + x_{3,4} > 1$, then by (\ref{eq-add1}) and Mathematica,
\[
\Gamma(G) \begin{cases}
> \Theta_{3,3}+ \Theta_{3,4} >  2\Theta_{2,4} & \text{for $-1\le\alpha<0$,}\\
< \Theta_{3,3}+ \Theta_{3,4} < 2\Theta_{2,4}  & \text{for $0<\alpha<1$,}\\
<  \Theta_{3,3}+ \Theta_{1,3} < 2\Theta_{1,3}+\Theta_{3,4} & \text{for $1<\alpha\le2$.}
\end{cases}
\]
If $x_{3,3} \ge 3$, then by (\ref{eq-add1}) and Mathematica,
\[
\Gamma(G) \begin{cases}
\ge 3\Theta_{3,3} >  2\Theta_{2,4} & \text{for $-1\le\alpha<0$,}\\
\le 3\Theta_{3,3} <  2\Theta_{2,4}  & \text{for $0<\alpha<1$,}\\
\le 3\Theta_{3,3} < 2\Theta_{1,3}+\Theta_{3,4} & \text{for $1<\alpha\le2$.}
\end{cases}
\]

%

In what follows, we assume that $x_{2,2}=x_{2,3}=x_{3,3}=0$. Then Eq. (\ref{Eq-K-00}) becomes
\begin{equation}\label{Ineq-lem-a}
\Gamma(G) = x_{1,2} \Theta_{1,2}+ x_{1,3}\Theta_{1,3}+ x_{2,4} \Theta_{2,4} + x_{3,4} \Theta_{3,4}.
\end{equation}
And from Eq. (\ref{Eq557b}), we know that
$x_{1,2}+ x_{2,4} = 2 n_2$ and $x_{1,3} + x_{3,4} = 3 n_3$.

Now by using Lemma \ref{Lem-K-01a} (i) and (ii) in Eq. (\ref{Ineq-lem-a}), we get
\begin{equation*}
\Gamma(G) \begin{cases}
\ge (x_{1,2}+ x_{2,4})\Theta_{2,4} + (x_{1,3} + x_{3,4})\Theta_{3,4}  & \text{for $-1\le\alpha<0$,}\\
\le (x_{1,2}+ x_{2,4})\Theta_{2,4} + (x_{1,3} + x_{3,4})\Theta_{3,4}  & \text{for $0<\alpha<1$,}\\
\le (x_{1,2}+ x_{2,4})\Theta_{1,2} + (x_{1,3} + x_{3,4})\Theta_{1,3}  & \text{for $1<\alpha\le2$,}
\end{cases}
\end{equation*}
i.e.,
\begin{equation}\label{Ineq-lem}
\Gamma(G) \begin{cases}
\ge 2 n_2\Theta_{2,4} + 3 n_3\Theta_{3,4}  & \text{for $-1\le\alpha<0$,}\\
\le 2 n_2\Theta_{2,4} + 3 n_3\Theta_{3,4}  & \text{for $0<\alpha<1$,}\\
\le 2 n_2\Theta_{1,2} + 3 n_3\Theta_{1,3}  & \text{for $1<\alpha\le2$.}
\end{cases}
\end{equation}


If $n_2\ge1$ and $n_3\ge1$, then by (\ref{Ineq-lem}) and Mathematica, we get
\[
\Gamma(G) \begin{cases}
\ge 2\Theta_{2,4} + 3\Theta_{3,4} > 2\Theta_{2,4}  & \text{for $-1\le\alpha<0$,}\\
\le 2\Theta_{2,4} + 3\Theta_{3,4} < 2\Theta_{2,4}  & \text{for $0<\alpha<1$,}\\
\le 2\Theta_{1,2} + 3\Theta_{1,3} \le  2\Theta_{1,3}+\Theta_{3,4}  & \text{for $1<\alpha\le2$.}
\end{cases}
\]



If $n_2=0$ and $n_3 \ge 2$ (note that $n_2 + n_3 \ge 2$ from hypothesis), then by (\ref{Ineq-lem}) and Mathematica, we get
\[
\Gamma(G) \begin{cases}
\ge 6\Theta_{3,4} > 2\Theta_{2,4}  & \text{for $-1\le\alpha<0$,}\\
\le 6\Theta_{3,4} < 2\Theta_{2,4}  & \text{for $0<\alpha<1$,}\\
\le 6\Theta_{1,3} <  2\Theta_{1,3}+\Theta_{3,4}  & \text{for $1<\alpha\le2$.}
\end{cases}
\]

%
%

If $n_2 \ge 2$ and $n_3=0$ (note that $n_2 + n_3 \ge 2$ from hypothesis), then by (\ref{Ineq-lem}) and Mathematica, we get
\[
\Gamma(G) \begin{cases}
\ge 4\Theta_{2,4} > 2\Theta_{2,4}  & \text{for $-1\le\alpha<0$,}\\
\le 4\Theta_{2,4} < 2\Theta_{2,4}  & \text{for $0<\alpha<1$.}
\end{cases}
\]
In particular, we note that the inequality $\Gamma(G)\le 4\Theta_{1,2} <  2\Theta_{1,3}+\Theta_{3,4}$ does not hold for $\alpha>\frac{43}{25}=1.72$. So we will prove the inequality $\Gamma(G)<2\Theta_{1,3}+\Theta_{3,4}$, for $1<\alpha\le2$, by another stricter way.

In the remaining proof, we assume that $n_2 \ge 2$, $n_3=0$, and $1<\alpha\le2$. From Eq. (\ref{Ineq-lem-a}), we get
\begin{equation}\label{Ineq-lem-b}
\Gamma(G)= x_{1,2}\Theta_{1,2} + x_{2,4}\Theta_{2,4}.
\end{equation}
Recall that $x_{1,2}+ x_{2,4} = 2 n_2 \ge 4$.
If $x_{1,2}=0$, then $x_{2,4} \ge 4$, and from Eq. (\ref{Ineq-lem-b}), it follows that $\Gamma(G)\le4\Theta_{2,4}<2\Theta_{1,3}+\Theta_{3,4}$.
If $x_{1,2}=1$, then $x_{2,4}\ge3$, and hence from Eq. (\ref{Ineq-lem-b}), we have $\Gamma(G)\le \Theta_{1,2}+ 3\Theta_{2,4}<2\Theta_{1,3}+\Theta_{3,4}$.
If $x_{1,2}\ge2$, then noting that $x_{2,4}\ge x_{1,2}$, i.e., $x_{2,4}\ge 2$, which together with Eq. (\ref{Ineq-lem-b}) and Mathematica imply that
\[
\Gamma(G)\le 2\Theta_{1,2}+2\Theta_{2,4}<2\Theta_{1,3}+\Theta_{3,4}.
\]

The proof is completed.
\end{proof}

\begin{thm}\label{cor44}
Let $G$ be a molecular $(n,m)$-graph, where $n-1\le m \le 2n$ and $n\ge5$.
\begin{enumerate}
\item[$(i)$]
Suppose that $m+n\equiv0$ (mod  3).
Then
\[
\chi_\al(G) \begin{cases}
\ge \frac{4}{3}(5^\al-8^\al)n-\frac{1}{3}(2\cdot 5^\al-5\cdot 8^\al)m & \text{for $-1\le\alpha<0$,}\\
\le \frac{4}{3}(5^\al-8^\al)n-\frac{1}{3}(2\cdot 5^\al-5\cdot 8^\al)m & \text{for $0<\alpha<1$ and $1<\alpha\le2$,}
\end{cases}
\]
with either equality if and only if $G$ contains no vertices of degrees $2$ and $3$ (i.e., $n_2 = 0$ and $n_3 = 0$).
\item[$(ii)$]
Suppose that $m+n\equiv1$ (mod  3).
For $-1\le\alpha<0$ and $0<\alpha<1$,
\[
\chi_\al(G) \begin{cases}
\ge \frac{4}{3}(5^\al-8^\al)n-\frac{1}{3}(2\cdot 5^\al-5\cdot 8^\al)m + 3\cdot7^\al -\frac{1}{3}\cdot 5^\al -\frac{8}{3}\cdot 8^\al & \text{for $-1\le\alpha<0$,}\\
\le \frac{4}{3}(5^\al-8^\al)n-\frac{1}{3}(2\cdot 5^\al-5\cdot 8^\al)m + 3\cdot7^\al -\frac{1}{3}\cdot 5^\al -\frac{8}{3}\cdot 8^\al & \text{for $0<\alpha<1$,}
\end{cases}
\]
with either equality if and only if $G$ contains no vertex of degree $2$ and exactly one vertex of degree $3$ (i.e., $n_2 = 0$ and $n_3 = 1$), which is adjacent to three vertices of degree $4$ (i.e., $x_{1,3} = 0$ and $x_{3,4} = 3$).
For $1<\alpha\le2$,
$$
\chi_\al(G) \le \frac{4}{3}(5^\al-8^\al)n-\frac{1}{3}(2\cdot 5^\al-5\cdot 8^\al)m + 2\cdot4^\al + 7^\al - \frac{7}{3}\cdot5^\al - \frac{2}{3}\cdot 8^\al,
$$
with equality if and only if $G$ contains no vertex of degree $2$ and exactly one vertex of degree $3$ (i.e., $n_2 = 0$ and $n_3 = 1$), which is adjacent to two vertices of degree $1$ and one vertex of degree $4$ (i.e., $x_{1,3} = 2$ and $x_{3,4} = 1$).
\item[$(iii)$]
Suppose that $m+n\equiv2$ (mod  3).
For $-1\le\alpha<0$ and $0<\alpha<1$,
\[
\chi_\al(G) \begin{cases}
\ge \frac{4}{3}(5^\al-8^\al)n-\frac{1}{3}(2\cdot 5^\al-5\cdot 8^\al)m + 2\cdot6^\al -\frac{2}{3}\cdot 5^\al -\frac{4}{3}\cdot 8^\al & \text{for $-1\le\alpha<0$,}\\
\le \frac{4}{3}(5^\al-8^\al)n-\frac{1}{3}(2\cdot 5^\al-5\cdot 8^\al)m + 2\cdot6^\al -\frac{2}{3}\cdot 5^\al -\frac{4}{3}\cdot 8^\al & \text{for $0<\alpha<1$,}
\end{cases}
\]
with either equality if and only if $G$ contains no vertex of degree $3$ and exactly one vertex of degree $2$ (i.e., $n_3 = 0$ and $n_2 = 1$), which is adjacent to two vertices of degree $4$ (i.e., $x_{1,2} = 0$ and $x_{2,4} = 2$).
For $1<\alpha\le2$,
$$
\chi_\al(G) \le \frac{4}{3}(5^\al-8^\al)n-\frac{1}{3}(2\cdot 5^\al-5\cdot 8^\al)m + 3^\al + 6^\al - \frac{5}{3}\cdot5^\al - \frac{1}{3}\cdot 8^\al,
$$
with equality if and only if $G$ contains no vertex of degree $3$ and exactly one vertex of degree $2$ (i.e., $n_3 = 0$ and $n_2 = 1$), which is adjacent to one vertex of degree $1$ and one vertex of degree $4$ (i.e., $x_{1,2} = 1$ and $x_{2,4} = 1$).
\end{enumerate}
\end{thm}

\begin{proof}
From Eqs. (\ref{Eq555b}) and (\ref{Eq556b}), the following congruence (see also \cite{Gutman2000}) follows:
\begin{equation}\label{Eq-K-01}
m+n\equiv n_3 - n_2 \text{ \ (mod  3)}.
\end{equation}

\textit{Case 1.} $m+n\equiv0$ (mod  3).




The desired result follows from Theorem \ref{thm-bd}.

\textit{Case 2.} $m+n\equiv1$ (mod  3).



First suppose that $n_2 + n_3\le1$. From (\ref{Eq-K-01}), there is only one condition: $n_2=0$ and $n_3=1$. Now
$$
\Gamma(G)= x_{1,3}\Theta_{1,3} + x_{3,4}\Theta_{3,4}
$$
and
$$
x_{1,3} + x_{3,4} = 3 n_3 = 3.
$$
From Lemma \ref{Lem-K-01a} (i), it follows that
\[
\Theta_{1,3} -\Theta_{3,4}
\begin{cases}
> 0  & \text{for $-1\le\alpha<0$ and $1<\alpha\le2$,}\\
< 0  & \text{for $0<\alpha<1$,}
\end{cases}
\]
which implies that
\[
\Gamma(G) \begin{cases}
\ge 3\Theta_{3,4}  & \text{for $-1\le\alpha<0$,}\\
\le 3\Theta_{3,4}  & \text{for $0<\alpha<1$,}\\
\le 2\Theta_{1,3}+\Theta_{3,4}  & \text{for $1<\alpha\le2$.}
\end{cases}
\]
The equality $\Gamma(G)=3\Theta_{3,4}$, for $-1\le\alpha<0$ and $0<\alpha<1$, holds if and only if the unique vertex of degree 3 is adjacent to three vertices of degree 4. The equality $\Gamma(G)=2\Theta_{1,3}+\Theta_{3,4}$, for $1<\alpha\le2$, holds if and only if the unique vertex of degree 3 is adjacent to two vertices of degree $1$ and one vertex of degree 4.

If $n_2 + n_3\ge2$, then from Lemma \ref{Lem-K-01} and Mathematica, it follows that
\[
\Gamma(G) \begin{cases}
> 2\Theta_{2,4} > 3\Theta_{3,4}  & \text{for $-1\le\alpha<0$,}\\
< 2\Theta_{2,4} < 3\Theta_{3,4}  & \text{for $0<\alpha<1$,}\\
< 2\Theta_{1,3}+\Theta_{3,4}  & \text{for $1<\alpha\le2$.}
\end{cases}
\]

Now combining the conclusions of both cases and using Eq. (\ref{Eq7}), we get the desired result.

\textit{Case 3.} $m+n\equiv2$ (mod  3).



First suppose that $n_2 + n_3\le1$. From (\ref{Eq-K-01}), there is only one condition: $n_2=1$ and $n_3=0$. Now
$$
\Gamma(G)= x_{1,2}\Theta_{1,2} + x_{2,4}\Theta_{2,4}
$$
and
$$
x_{1,2} + x_{2,4} = 2 n_2 = 2.
$$
From Lemma \ref{Lem-K-01a} (ii), it follows that
\[
\Theta_{1,2} -\Theta_{2,4}
\begin{cases}
> 0  & \text{for $-1\le\alpha<0$ and $1<\alpha\le2$,}\\
< 0  & \text{for $0<\alpha<1$,}
\end{cases}
\]
which implies that
\[
\Gamma(G) \begin{cases}
\ge 2\Theta_{2,4}  & \text{for $-1\le\alpha<0$,}\\
\le 2\Theta_{2,4}  & \text{for $0<\alpha<1$,}\\
\le \Theta_{1,2}+\Theta_{2,4}  & \text{for $1<\alpha\le2$.}
\end{cases}
\]
The equality $\Gamma(G)=2\Theta_{2,4}$, for $-1\le\alpha<0$ and $0<\alpha<1$, holds if and only if the unique vertex of degree 2 is adjacent to two vertices of degree 4. The equality $\Gamma(G)=\Theta_{1,2}+\Theta_{2,4}$, for $1<\alpha\le2$, holds if and only if the unique vertex of degree 2 is adjacent to one vertex of degree $1$ and one vertex of degree 4.


If $n_2 + n_3\ge2$, then from Lemma \ref{Lem-K-01} and Mathematica, we have
\[
\Gamma(G) \begin{cases}
> 2\Theta_{2,4}  & \text{for $-1\le\alpha<0$,}\\
< 2\Theta_{2,4}  & \text{for $0<\alpha<1$,}\\
< 2\Theta_{1,3}+\Theta_{3,4} < \Theta_{1,2}+\Theta_{2,4}    & \text{for $1<\alpha\le2$.}
\end{cases}
\]

Again, after combining the conclusions of both cases and using Eq. (\ref{Eq7}), we get the required result.

Combining the above three cases, the result follows.
\end{proof}

As illustrations of Theorem \ref{cor44}, we list some molecular $(n,m)$-graphs, in particular when $m = n - 1, n, n+1$, whose $\chi_\al$ values attain the bounds established in Theorem \ref{cor44}, see Figure \ref{fig:5}.

\begin{figure}[ht]
	\centering
	\includegraphics[width=1\textwidth]{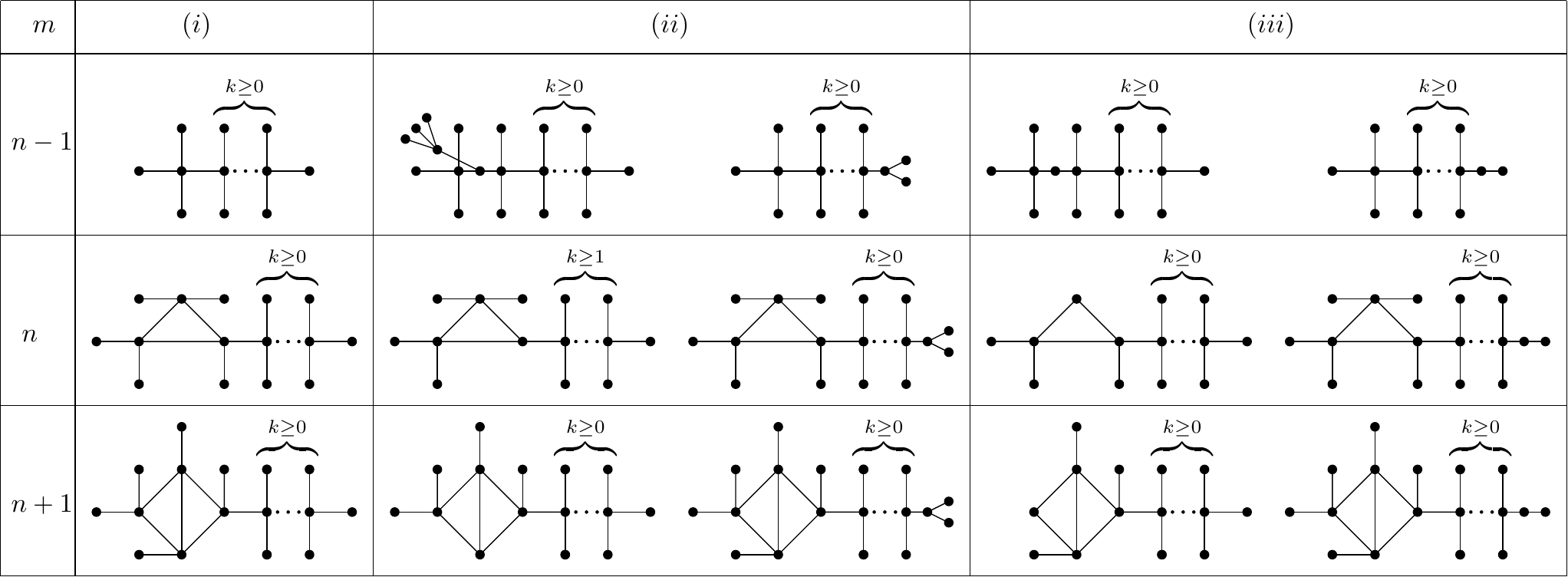}
	\caption{Some examples of the extremal $(n,m)$-graphs, specified in Theorem \ref{cor44}, when $m=n-1,n,n+1$.}
  \label{fig:5}	
  \end{figure}

\begin{rmk}
In case of trees, it holds that $m=n-1$ and hence Theorem \ref{cor44} generalizes a result of \cite{Zhou09} concerning sum-connectivity index $\chi_{-1/2}$ as well as a result of \cite{Liu14} concerning harmonic index $2\chi_{-1}$. Theorem \ref{cor44} also gives the sharp upper bound of the hyper Zagreb index $\chi_2$ for molecular $(n,m)$-graphs.
\end{rmk}

Similar to Eq. (\ref{Eq557bc}), we have the following equation concerning the general Platt index $Pl_{\alpha}$:
\begin{eqnarray}\label{Eq557bc-Pl}
Pl_\al(G)&=& \frac{4}{3}(3^\al-6^\al)n-\frac{1}{3}(2\cdot 3^\al-5\cdot 6^\al)m  + x_{1,2}\Theta'_{1,2}
+x_{1,3}\Theta'_{1,3} + x_{2,2}\Theta'_{2,2}\nonumber \\
&& + x_{2,3}\Theta'_{2,3} +x_{2,4}\Theta'_{2,4}+x_{3,3}\Theta'_{3,3}+x_{3,4}\Theta'_{3,4},
\end{eqnarray}
where
$$
\Theta'_{1,2}= 1 -\frac{4}{3}\cdot 3^\al +\frac{1}{3}\cdot 6^\al, \ \
\Theta'_{1,3} = 2^\al -\frac{10}{9}\cdot 3^\al +\frac{1}{9}\cdot 6^\al,
$$
$$
\Theta'_{2,2}= 2^\al -\frac{2}{3}\cdot 3^\al -\frac{1}{3}\cdot 6^\al, \ \
\Theta'_{2,3}= \frac{5}{9} \left( 3^\al - 6^\al \right),
$$
$$
\Theta'_{2,4}= 4^\al -\frac{1}{3}\cdot 3^\al -\frac{2}{3}\cdot 6^\al, \ \
\Theta'_{3,3}= 4^\al -\frac{2}{9}\cdot 3^\al -\frac{7}{9}\cdot 6^\al,
$$
$$
\Theta'_{3,4}= 5^\al -\frac{1}{9}\cdot 3^\al -\frac{8}{9}\cdot 6^\al.
$$

It can be easily checked that $ \Theta'_{1,3}, \Theta'_{2,2}, \Theta'_{2,3}, \Theta'_{2,4}, \Theta'_{3,3}$ and $\Theta'_{3,4}$ are all positive for $-1\le \al <0$ and all negative for $0< \al \le 2$. The case about $\Theta'_{1,2}$ is somewhat different. In $[-1,0) \cup (0,2]$, there is a unique root of the equation $\Theta'_{1,2}=0$ on $\al$, which is $x_0\approx1.8509$. More precisely,
$\Theta'_{1,2}$ is positive for $-1\le \al <0$ and $x_0 < \al \le 2$, and negative for $0< \al < x_0$.
Instead, $\Theta'_{1,2}+\Theta'_{2,2}$, $\Theta'_{1,2}+\Theta'_{2,3}$ and $\Theta'_{1,2}+\Theta'_{2,4}$ are all negative for $x_0 \le \al \le 2$.


It is clear that the inequality $x_{1,2} \le x_{2,2}+x_{2,3}+x_{2,4}$ holds for every molecular graph with at least $5$ vertices. Hence,
\begin{equation}\label{Eq557bc-Pl-L}
\Gamma'(G)=x_{1,2}\Theta'_{1,2} +x_{1,3}\Theta'_{1,3} + x_{2,2}\Theta'_{2,2}+ x_{2,3}\Theta'_{2,3} +x_{2,4}\Theta'_{2,4}+x_{3,3}\Theta'_{3,3}+x_{3,4}\Theta'_{3,4}
\end{equation}
is non-negative for $-1\le \al <0$ and non-positive for $0< \al \le 2$, and therefore we have the next result.

\begin{thm}\label{thm-pl}
If $G$ is a molecular $(n,m)$-graph, $n\ge5$, then
\[
Pl_\al(G) \begin{cases}
\ge \frac{4}{3}(3^\al-6^\al)n-\frac{1}{3}(2\cdot 3^\al-5\cdot 6^\al)m & \text{for $-1\le\alpha<0$,}\\
\le \frac{4}{3}(3^\al-6^\al)n-\frac{1}{3}(2\cdot 3^\al-5\cdot 6^\al)m & \text{for $0<\alpha\le2$,}
\end{cases}
\]
with either equality if and only if $G$ does not contain vertices of degrees 2 and 3.
\end{thm}

Recall that $x_0\approx1.8509$ is the unique root of the equation $\Theta'_{1,2}=0$ in $[-1,0) \cup (0,2]$. By virtue of Mathematica, Lemma \ref{Lem-K-01a} remains true, if we replace $\Theta_{i,j}$ with $\Theta'_{i,j}$, with an exception that
$\max \{ \Theta'_{2,2} , \Theta'_{2,3} , \Theta'_{2,4} \} < \Theta'_{1,2}<0$ when $x_0 \le \alpha \le 2$, which should be corrected into $\max \{ \Theta'_{2,2}, \Theta'_{2,3}, \Theta'_{2,4} \} < 0 \le \Theta'_{1,2}$ for $x_0\le\alpha\le2$.


Suppose that $x_{2,2}=x_{2,3}=x_{3,3}=0$. Recall that $\Theta'_{1,2}+\Theta'_{2,4} < 0$ when $x_0\le\alpha\le2$. If $n_2\ge1$ and $n_3\ge1$, then it holds that $x_{2,4}\ge x_{1,2}$ and $x_{2,4}\ge1$,
hence
\begin{eqnarray*}
\Gamma'(G)&\le& x_{2,4}(\Theta'_{1,2}+\Theta'_{2,4})+(x_{1,3}+x_{3,4})\Theta'_{1,3} \\
&=& x_{2,4}(\Theta'_{1,2}+\Theta'_{2,4})+3 n_3\Theta'_{1,3} \\
&\le& (\Theta'_{1,2}+\Theta'_{2,4}) + 3\Theta'_{1,3}\\
&<& 2\Theta'_{1,3} +\Theta'_{3,4}
\end{eqnarray*}
for $x_0\le\alpha\le2$.
If $n_2 \ge 2,  n_3 = 0$ and $x_{1,2}\ge2$, then noting that $x_{2,4}\ge x_{1,2}$, i.e., $x_{2,4}\ge 2$, which implies that
\[
\Gamma'(G)= x_{1,2}\Theta'_{1,2}+x_{2,4}\Theta'_{2,4} \le x_{2,4}(\Theta'_{1,2}+\Theta'_{2,4}) \le 2(\Theta'_{1,2}+\Theta'_{2,4})<2\Theta'_{1,3}+\Theta'_{3,4}
\]
for $x_0\le\alpha\le2$.
For the remaining cases, proof of the following lemma is fully analogous to that of Lemma \ref{Lem-K-01} and hence omitted:

\begin{lem} \label{Lem-K-01-PL}
If a molecular graph $G$ satisfies the inequality $n_2+n_3 \ge 2$, then for the invariant $\Gamma'(G)$, defined in Eq. (\ref{Eq557bc-Pl-L}), it holds that
\[
\Gamma'(G) \begin{cases}
> 2\Theta'_{2,4}  & \text{for $-1\le\alpha<0$,}\\
< 2\Theta'_{2,4}  & \text{for $0<\alpha<1$,}\\
< 2\Theta'_{1,3}+\Theta'_{3,4}   & \text{for $1<\alpha\le2$.}
\end{cases}
\]
\end{lem}

The proof of the next result is fully analogous to that of Theorem \ref{cor44} and hence omitted.

\begin{thm}\label{cor44-pl}
Let $G$ be a molecular $(n,m)$-graph, where $n-1\le m \le 2n$ and $n\ge5$.
\begin{enumerate}
\item[$(i)$]
Suppose that $m+n\equiv0$ (mod  3).
Then
\[
Pl_\al(G) \begin{cases}
\ge \frac{4}{3}(3^\al-6^\al)n-\frac{1}{3}(2\cdot 3^\al-5\cdot 6^\al)m & \text{for $-1\le\alpha<0$,}\\
\le \frac{4}{3}(3^\al-6^\al)n-\frac{1}{3}(2\cdot 3^\al-5\cdot 6^\al)m & \text{for $0<\alpha<1$ and $1<\alpha\le2$,}
\end{cases}
\]
with either equality if and only if $G$ contains no vertices of degrees $2$ and $3$ (i.e., $n_2 = 0$ and $n_3 = 0$).
\item[$(ii)$]
Suppose that $m+n\equiv1$ (mod  3).
For $-1\le\alpha<0$ and $0<\alpha<1$,
\[
Pl_\al(G) \begin{cases}
\ge \frac{4}{3}(3^\al-6^\al)n-\frac{1}{3}(2\cdot 3^\al-5\cdot 6^\al)m + 3\cdot5^\al -\frac{1}{3}\cdot 3^\al -\frac{8}{3}\cdot 6^\al & \text{for $-1\le\alpha<0$,}\\
\le \frac{4}{3}(3^\al-6^\al)n-\frac{1}{3}(2\cdot 3^\al-5\cdot 6^\al)m + 3\cdot5^\al -\frac{1}{3}\cdot 3^\al -\frac{8}{3}\cdot 6^\al & \text{for $0<\alpha<1$,}
\end{cases}
\]
with either equality if and only if $G$ contains no vertex of degree $2$ and exactly one vertex of degree $3$ (i.e., $n_2 = 0$ and $n_3 = 1$), which is adjacent to three vertices of degree $4$ (i.e., $x_{1,3} = 0$ and $x_{3,4} = 3$).
For $1<\alpha\le2$,
$$
Pl_\al(G) \le \frac{4}{3}(3^\al-6^\al)n-\frac{1}{3}(2\cdot 3^\al-5\cdot 6^\al)m + 2\cdot2^\al + 5^\al - \frac{7}{3}\cdot3^\al - \frac{2}{3}\cdot 6^\al,
$$
with equality if and only if $G$ contains no vertex of degree $2$ and exactly one vertex of degree $3$ (i.e., $n_2 = 0$ and $n_3 = 1$), which is adjacent to two vertices of degree $1$ and one vertex of degree $4$ (i.e., $x_{1,3} = 2$ and $x_{3,4} = 1$).
\item[$(iii)$]
Suppose that $m+n\equiv2$ (mod  3).
For $-1\le\alpha<0$ and $0<\alpha<1$,
\[
Pl_\al(G) \begin{cases}
\ge \frac{4}{3}(3^\al-6^\al)n-\frac{1}{3}(2\cdot 3^\al-5\cdot 6^\al)m + 2\cdot4^\al -\frac{2}{3}\cdot 3^\al -\frac{4}{3}\cdot 6^\al & \text{for $-1\le\alpha<0$,}\\
\le \frac{4}{3}(3^\al-6^\al)n-\frac{1}{3}(2\cdot 3^\al-5\cdot 6^\al)m + 2\cdot4^\al -\frac{2}{3}\cdot 3^\al -\frac{4}{3}\cdot 6^\al & \text{for $0<\alpha<1$,}
\end{cases}
\]
with either equality if and only if $G$ contains no vertex of degree $3$ and exactly one vertex of degree $2$ (i.e., $n_3 = 0$ and $n_2 = 1$), which is adjacent to two vertices of degree $4$ (i.e., $x_{1,2} = 0$ and $x_{2,4} = 2$).
For $1<\alpha\le2$,
$$
Pl_\al(G) \le \frac{4}{3}(3^\al-6^\al)n-\frac{1}{3}(2\cdot 3^\al-5\cdot 6^\al)m + 1 + 4^\al - \frac{5}{3}\cdot3^\al - \frac{1}{3}\cdot 6^\al,
$$
with equality if and only if $G$ contains no vertex of degree $3$ and exactly one vertex of degree $2$ (i.e., $n_3 = 0$ and $n_2 = 1$), which is adjacent to one vertex of degree $1$ and one vertex of degree $4$ (i.e., $x_{1,2} = 1$ and $x_{2,4} = 1$).
\end{enumerate}
\end{thm}

In Theorem \ref{cor44-pl}, we considered all non-zero values of $\al$ between $-1$ and $2$, except $\al = 1$, because the sharp bounds of $Pl_1$ for molecular $(n,m)$-graphs follows from the results concerning $M_1$, established in \cite{Hu05}. Theorem \ref{cor44-pl} also gives sharp bounds on the reformulated first Zagreb index $EM_1$ for molecular $(n,m)$-graphs.

Finally, similar to Eq. (\ref{Eq557bc}), we have the following equation concerning the ordinary generalized geometric-arithmetic index $\mbox{OGA}_{k}$:
\begin{eqnarray}\label{Eq557bc-GA}
\mbox{OGA}_{k}(G)&=& \frac{4}{3} \left(\left(\frac{4}{5}\right)^k - 1\right) n - \frac{1}{3} \left( 2\left(\frac{4}{5}\right)^k - 5 \right) m
+ \Upsilon(G),
\end{eqnarray}
where
\begin{equation}\label{Eq557bc-GA-1}
\Upsilon(G)= x_{1,2}\Phi_{1,2}
+x_{1,3}\Phi_{1,3} + x_{2,2}\Phi_{2,2} + x_{2,3}\Phi_{2,3} +x_{2,4}\Phi_{2,4}+x_{3,3}\Phi_{3,3}+x_{3,4}\Phi_{3,4},
\end{equation}
and
$$
\Phi_{1,2}= \left( \frac{2\sqrt{2}}{3} \right)^k - \frac{4}{3} \left( \frac{4}{5} \right)^k + \frac{1}{3}, \ \
\Phi_{1,3} = \left( \frac{\sqrt{3}}{2} \right)^k - \frac{10}{9} \left( \frac{4}{5} \right)^k + \frac{1}{9},
$$
$$
\Phi_{2,2}= \frac{2}{3} \left( 1 - \left( \frac{4}{5} \right)^k \right), \ \
\Phi_{2,3}= \left( \frac{2\sqrt{6}}{5} \right)^k - \frac{4}{9} \left( \frac{4}{5} \right)^k - \frac{5}{9},
$$
$$
\Phi_{2,4}= \left( \frac{2\sqrt{2}}{3} \right)^k - \frac{1}{3} \left( \frac{4}{5} \right)^k - \frac{2}{3}, \ \
\Phi_{3,3}= \frac{2}{9} \left( 1 - \left( \frac{4}{5} \right)^k \right),
$$
$$
\Phi_{3,4}= \left( \frac{4\sqrt{3}}{7} \right)^k - \frac{1}{9} \left( \frac{4}{5} \right)^k - \frac{8}{9}.
$$
It can be easily checked that the inequalities
$$
\Phi_{1,2}>\Phi_{2,2}>\Phi_{2,3}>\Phi_{2,4}>0
$$
and
$$
\Phi_{1,3}>\Phi_{2,3}>\Phi_{3,3}>\Phi_{3,4}>0
$$
hold for $0<k\le1$, which implies the following lemma (whose proof is fully analogous to that of Lemma \ref{Lem-K-01} and hence omitted):

\begin{lem} \label{Lem-K-01-GA}
Let $0<k\le1$ and $G$ be a molecular graph satisfying the inequality $n_2+n_3 \ge 2$. For the invariant $\Upsilon(G)$, defined in Eq. (\ref{Eq557bc-GA-1}), it holds that $ \Upsilon(G) > 3\Phi_{3,4}$.
\end{lem}

The proof of the following result is similar to that of Theorem \ref{cor44} and hence omitted:


\begin{thm}\label{cor44-GA}
Let $G$ be a molecular $(n,m)$-graph, where $n-1\le m \le 2n$ and $n\ge5$.
Suppose that $0<k\le1$.
\begin{enumerate}
\item[$(i)$]
If $m+n\equiv0$ (mod  3), then
\[
\mbox{OGA}_{k} (G) \ge \frac{4}{3} \left(\left(\frac{4}{5}\right)^k - 1\right) n - \frac{1}{3} \left( 2\left(\frac{4}{5}\right)^k - 5 \right) m,
\]
with equality if and only if $G$ contains no vertices of degrees $2$ and $3$ (i.e., $n_2 = 0$ and $n_3 = 0$).
\item[$(ii)$]
If $m+n\equiv1$ (mod  3), then
\[
\mbox{OGA}_{k} (G) \ge \frac{4}{3} \left(\left(\frac{4}{5}\right)^k - 1\right) n - \frac{1}{3} \left( 2\left(\frac{4}{5}\right)^k - 5 \right) m + 3\left( \frac{4\sqrt{3}}{7} \right)^k - \frac{1}{3} \left( \frac{4}{5} \right)^k - \frac{8}{3},
\]
with equality if and only if $G$ contains no vertex of degree $2$ and exactly one vertex of degree $3$ (i.e., $n_2 = 0$ and $n_3 = 1$), which is adjacent to three vertices of degree $4$ (i.e., $x_{1,3} = 0$ and $x_{3,4} = 3$).
\item[$(iii)$]
If $m+n\equiv2$ (mod  3), then
\[
\mbox{OGA}_{k} (G) \ge \frac{4}{3} \left(\left(\frac{4}{5}\right)^k - 1\right) n - \frac{1}{3} \left( 2\left(\frac{4}{5}\right)^k - 5 \right) m + 2\left( \frac{2\sqrt{2}}{3} \right)^k - \frac{2}{3} \left( \frac{4}{5} \right)^k - \frac{4}{3},
\]
with equality if and only if $G$ contains no vertex of degree $3$ and exactly one vertex of degree $2$ (i.e., $n_3 = 0$ and $n_2 = 1$), which is adjacent to two vertices of degree $4$ (i.e., $x_{1,2} = 0$ and $x_{2,4} = 2$).
\end{enumerate}
\end{thm}

\noindent {\bf Acknowledgement.} This work was supported by the National Natural Science Foundation of China (Grant No. 11701505).


\begin{thebibliography}{00}

\bibitem{Akhter17} S. Akhter, M. Imran, Z. Raza, Bounds for the general sum-connectivity index of composite graphs, \textit{J. Inequal. Appl.} \textbf{2017} (2017) 76.


\bibitem{Atabati17} M. Atabati, R. Emamalizadeh, A quantitative structure property relationship for prediction of flash point of alkanes using molecular connectivity indices, \textit{Chinese J. Chem. Engi.} \textbf{21} (2013) 420.



\bibitem{Ali17} A. Ali, An alternative but short proof of a result of Zhu and Lu concerning general sum-connectivity index, \textit{Asian-European J. Math.} \textbf{11}(2) (2018) 1850030.

\bibitem{Ali-18} A. Ali, D. Dimitrov, On the extremal graphs with respect to bond incident degree indices, \textit{Discrete Appl. Math.} \textbf{238} (2018) 32--40.

\bibitem{Ali-Dimitov-2018} A. Ali, D. Dimitrov, Z. Du, F. Ishfaq, On the extremal graphs for general sum-connectivity index ($\chi_{_\alpha}$) with given cyclomatic number when $\alpha>1$, \textit{Discrete Appl. Math.} DOI: 10.1016/j.dam.2018.10.009, in press.

\bibitem{Ali17R} A. Ali, Z. Du, On the difference between atom-bond connectivity index and Randi\'{c} index of binary and chemical trees, \textit{Int. J. Quantum Chem.} \textbf{117} (2017) e25446.

\bibitem{Ali18}	A. Ali, I. Gutman, E. Milovanovi\'{c}, I. Milovanovi\'{c}, Sum of powers of the degrees of graphs: extremal results and bounds, \textit{MATCH Commun. Math. Comput. Chem.} {\bf 80} (2018) 5--84.

\bibitem{Ali-MATCH-19} A. Ali, L. Zhong, I. Gutman, Harmonic index and its generalizations: extremal results and bounds, \textit{MATCH Commun. Math. Comput. Chem.} \textbf{81} (2019) 249--311.

\bibitem{Arshad17} M. Arshad, I. Tomescu, Maximum general sum-connectivity index with $-1\le \al <0$ for bicyclic graphs, \textit{Math. Reports} \textbf{19}  (2017) 93--96.

    \bibitem{Bondy08} J. A. Bondy, U. S. R. Murty, {\it Graph Theory},
        Springer, London, 2008.


\bibitem{Borovicanin-17} B. Borovi\'canin, K. C. Das, B. Furtula, I. Gutman,
        Bounds for Zagreb indices, {\it MATCH Commun. Math. Comput. Chem.\/}
        {\bf 78} (2017) 17--100.


\bibitem{Borovicanin-MCM19} B. Borovi\'canin, K. C. Das, B. Furtula, I. Gutman,
        Zagreb indices: Bounds and extremal graphs, in: I. Gutman, B. Furtula,
        K. C. Das, E. Milovanovi\'c, I. Milovanovi\'c (Eds.),
        {\it Bounds in Chemical Graph Theory -- Basics\/}, Univ. Kragujevac,
        Kragujevac, 2017, pp. 67--153.





\bibitem{Cui-17} Q. Cui, L. Zhong, On the general sum-connectivity index of trees with given number of pendent vertices, \textit{Discrete Appl. Math.} \textbf{222} (2017) 213--221.


\bibitem{Cui17} Q. Cui, L. Zhong, The general Randi\'{c} index of trees with given number of pendent vertices, \textit{Appl. Math. Comput.} \textbf{302} (2017) 111--121.

\bibitem{Das-17} K. C. Das, S. Balachandran, I. Gutman, Inverse degree, Randi\'{c} index and harmonic index of graphs, \textit{Appl. Anal. Discr. Math.} \textbf{11} (2017) 304--313.


\bibitem{Doslic11} T. Do\v{s}li\'{c}, B. Furtula, A. Graovac, I. Gutman, S. Moradi, Z. Yarahmadi, On vertex-degree-based molecular structure descriptors, \textit{MATCH Commun. Math. Comput. Chem.} \textbf{66} (2011) 613--626.

\bibitem{Eliasi11} M. Eliasi, A. Iranmanesh, On ordinary generalized geometric--arithmetic index, \textit{Appl. Math. Lett.} \textbf{24} (2011) 582--587.

\bibitem{Fajtlowicz87} S. Fajtlowicz, On conjectures of Graffiti-II,\textit{ Congr. Numer.} \textbf{60} (1987) 187--197.

\bibitem{Gutman-18} I. Gutman, B. Furtula, V. Katani\'{c}, Randi\'{c} index and information, \textit{AKCE Int. J. Graphs Comb.} DOI:10.1016/j.akcej.2017.09.006, in press.


\bibitem{Gutman2000} I. Gutman, O. Miljkovi\'{c}, Molecules with smallest connectivity indices, \textit{MATCH Commun. Math. Comput. Chem.} \textbf{41} (2000) 57--70.

\bibitem{Gutman86} I. Gutman, O. E. Polansky, {\it Mathematical Concepts
        in Organic Chemistry}, Springer, Berlin, 1986.

\bibitem{Gutman72} I. Gutman, N. Trinajsti\'{c}, Graph theory and molecular orbitals. Total $\pi$-electron energy of alternant hydrocarbons, \textit{Chem. Phys. Lett.} \textbf{17} (1972) 535--538.

\bibitem{Harary69} F. Harary, {\it Graph Theory\/}, Addison--Wesley, Reading, 1969.

\bibitem{Hu05} Y. Hu, X. Li, Y. Shi, T. Xu, I. Gutman, On molecular graphs with smallest and greatest zeroth-order general Randi\'{c} index, \textit{MATCH Commun. Math. Comput. Chem.} \textbf{54} (2005) 425--434.

\bibitem{Jamil17} M. K. Jamil, I. Tomescu, Minimum general sum-connectivity index of trees and unicyclic graphs having a given matching number, \textit{Discrete Appl. Math.} \textbf{222} (2017) 143--150.

\bibitem{Li16} F. Li, Q. Ye, The general connectivity indices of fluoranthene-type benzenoid systems, \textit{Appl. Math. Comput.} \textbf{273} (2016) 897--911.

\bibitem{Li08} X. Li, Y. Shi, A survey on the Randi\'{c} index, \textit{MATCH Commun. Math. Comput. Chem.} \textbf{59} (2008) 127--156.


\bibitem{Liu14} S. Liu, J. Li, Some properties on the harmonic index of molecular trees, \textit{ISRN Appl. Math.} \textbf{2014} (2014) 781648.


\bibitem{Lucic-13} B. Lu\v{c}i\'c, I. Sovi\'c, J. Batista, K. Skala, D. Plav\v{s}i\'c, D. Viki\'c-Topi\'c, D. Be\v{s}lo, S. Nikoli\'c, N. Trinajsti\'c, The sum--connectivity index -- an additive variant of the Randi\'c connectivity index, {\it Curr. Comput. Aided Drug Des.} {\bf 9} (2013) 184--194.

\bibitem{Lucic-15} B. Lu\v{c}i\'c, I. Sovi\'c, N. Trinajsti\'c, The four connectivity matrices, their indices, polynomials and spectra, in: S. C. Basak, G. Restrepo, J. L. Villaveces (Eds.), {\it Advances in Mathematical Chemistry and Applications}, Vol. 1 (Revised Edition), Elsevier, 2015, pp. 76--91.

\bibitem{Lucic-Trinajstic-Zhou-CPL-09} B. Lu\v{c}i\'c, N. Trinajsti\'c, B. Zhou, Comparison between the sum--connectivity index and product--connectivity index for benzenoid hydrocarbons, {\it Chem. Phys. Lett.} {\bf 475} (2009) 146--148.

\bibitem{Mansour17} T. Mansour, M. A. Rostami, S. Elumalai, B. A. Xavier, Correcting a paper on the Randi\'{c} and geometric-arithmetic indices, \textit{Turk. J. Math.} \textbf{41} (2017) 27--32.

\bibitem{Mili04} A. Mili\v{c}evi\'{c}, S. Nikoli\'{c}, N. Trinajsti\'{c}, On reformulated Zagreb indices, \textit{Mol. Divers.} \textbf{8} (2004) 393--399.

\bibitem{Mozrzymas17} A. Mozrzymas, Molecular connectivity indices for modeling the critical micelle concentration of cationic (chloride) Gemini surfactants, \textit{Colloid Polym. Sci.} \textbf{295} (2017) 75--87.




\bibitem{Platt52} J. R. Platt, Prediction of isomeric differences in paraffin properties, \textit{J. Phys. Chem.} \textbf{56} (1952) 328--336.

\bibitem{Ramane17} H. S. Ramane, V. V. Manjalapur, I. Gutman, General sum-connectivity index, general product-connectivity index, general Zagreb index and coindices of line graph of subdivision graphs, \textit{AKCE Int. J. Graphs Comb.} \textbf{14} (2017) 92--100.


\bibitem{r3} M. Randi\'{c}, On characterization of molecular branching, \textit{J. Am. Chem. Soc.} \textbf{97} (1975) 6609--6615.

\bibitem{Santos17} B. M. Santos, J. L. Zotin, P. M. Ndiaye, M. A. P. da Silva, Hydrogenation of aromatics compounds: Calculation of thermodynamic properties using molecular index connectivity, \textit{Can. J. Chem. Eng.} \textbf{95} (2017) 2272--2277.


\bibitem{Shafiei17} F. Shafiei, Prediction of physical and thermodynamic properties of aliphatic ethers from molecular structures by multiple linear regression, \textit{J. Chil. Chem. Soc.} \textbf{62} (2017) 3389--3392.


\bibitem{Shirdel13} G. H. Shirdel, H. Rezapour, A. M. Sayadi, The hyper-Zagreb index of graph operations, \textit{Iranian J. Math. Chem.} \textbf{4} (2013) 213--220.

\bibitem{ToCo2000} R. Todeschini, V. Consonni, \textit{Handbook of Molecular Descriptors}, Wiley-VCH, Weinheim, (2000).


\bibitem{Tomescu16} I. Tomescu, On the general sum-connectivity index of connected graphs with given order and girth,
\textit{Electron. J. Graph Theory Appl.}
\textbf{4} (2016) 1--7.

\bibitem{Trinajstic92} N. Trinajsti\'c, {\it Chemical Graph Theory\/}, CRC Press,
        Boca Raton, 1992.


\bibitem{Vukicevic14} \v{Z}. K. Vuki\'{c}evi\'{c}, G. Popivoda, Chemical trees with extreme values of Zagreb indices and coindices, \textit{Iranian J. Math. Chem.} \textbf{5} (2014) 19--29.

\bibitem{Vukcevic-Trinajstic-CCA-11} D. Vuki\v{c}evi\'c, N. Trinajsti\'c, Bond–additive modeling. 3. Comparison between the product--connectivity index and sum--connectivity index, {\it Croat. Chem. Acta} {\bf 83} (2010) 349--351.

\bibitem{Wang17} H. Wang, J.-B. Liu, S. Wang, W. Gao, S. Akhter, M. Imran, M. R. Farahani, Sharp bounds for the general sum-connectivity indices of transformation graphs,  \textit{Discrete Dyn. Nat. Soc.} \textbf{2017} (2017) 2941615.


\bibitem{Zhou09} B. Zhou, N. Trinajsti\'{c}, On a novel connectivity index, \textit{J. Math. Chem.} \textbf{46} (2009) 1252--1270.

\bibitem{Zhou10} B. Zhou, N. Trinajsti\'{c}, On general sum-connectivity index, \textit{J. Math. Chem.} \textbf{47} (2010) 210--218.

\bibitem{Zhu16} Z. Zhu, H. Lu, On the general sum-connectivity index of tricyclic graphs, \textit{J. Appl. Math. Comput.} \textbf{51} (2016) 177--188.










\end{thebibliography}
\end{document}